\documentclass[12pt]{amsart}
\usepackage{amssymb,mathrsfs}
\usepackage{amsthm,amsmath,amssymb}
\usepackage[numbers,sort&compress]{natbib}

\makeatletter

\newcommand{\Rmnum}[1]{\expandafter\@slowromancap\romannumeral #1@}
\makeatother
\usepackage{mathtools}

\theoremstyle{plain}
\newtheorem{theorem}{Theorem}[section]

\newtheorem{proposition}[theorem]{Proposition}
\newtheorem{corollary}[theorem]{Corollary}

\theoremstyle{definition}

\newtheorem{example}[theorem]{Example}

\newcommand{\abs}[1]{\lvert#1\rvert}

\usepackage[dvips]{graphics}
\usepackage{mathrsfs}
\usepackage{amssymb}
\usepackage{stmaryrd}
\usepackage{pifont}
\usepackage{amsmath}
\usepackage{hyperref}
\usepackage{amsthm}
\usepackage{cases}
\usepackage{mathtools}
\allowdisplaybreaks 
\title[UC and APP]{Unbounded convergences in Banach lattices and applications}

\date{\today}


\author[Z. Wang]{Zhangjun Wang$^{1}$}
\address{$^1$ The first author:School of Mathematics, Southwest Jiaotong University,
	Chengdu, Sichuan,
	China, 610000.}
\email{zhangjunwang@my.swjtu.edu.cn}

\author[Z. Chen]{Zili Chen$^{2}$}
\address{$^2$ The second author:School of Mathematics, Southwest Jiaotong University, Chengdu, Sichuan,
	China, 610000.}
\email{zlchen@home.swjtu.edu.cn}

\author[J. Chen]{Jinxi Chen$^{3}$}
\address{$^3$ The third author: School of Mathematics, Southwest Jiaotong University, Chengdu, Sichuan,
	China, 610000.}
\email{jinxichen@home.swjtu.edu.cn}

\begin{document}

\begin{abstract}
Several recent papers investigated unbounded convergences in Banach lattices. Combine all unbounded convergences, including \emph{unbounded order (norm, absolute weak, absolute weak*) convergence}, we characterize L-weakly compact sets, L-weakly compact operators and M-weakly compact operators on Banach lattices. Some related results are obtained as well.
\end{abstract}
	
\maketitle

\section{Introduction}
A net $(x_\alpha)_{\alpha\in A}$ in a Riesz space $E$ is \emph{order convergent} to $x\in E$ (write $x_\alpha\xrightarrow{o}x$) if there exists a net $(y_\beta)$, possibly over a different index set, such that $y_\beta\downarrow0$ and for each $\beta\in B$ there exists $\alpha_0\in A$ satisfying $|x_\alpha-x|\leq y_\beta$ for all $\alpha\geq \alpha_0$. The \emph{unbounded order convergence} is considered firstly by Nakano in \cite{N:48} and introduced by Wickstead in \cite{W:77}. A net $(x_\alpha)$ in a Banach lattice $E$ is unbounded order (resp. norm, absolute weak) convergent to some $x$, denoted by $x_\alpha\xrightarrow{uo}x$ (resp. $x_\alpha\xrightarrow{un}x$, $x_\alpha\xrightarrow{uaw}x$), if the net $(|x_\alpha-x|\wedge u)$ converges to zero in order (resp. norm, weak) for all $u\in E_+$. A net $(x_\alpha^\prime)$ in a dual Banach lattice $E^\prime$ is unbounded absolute weak* convergent to some $x^\prime$, denoted by $x_\alpha^\prime\xrightarrow{uaw^*}x^\prime$, if $|x^\prime_\alpha-x^\prime|\wedge u^\prime\xrightarrow{w^*}0$ for all $u^\prime\in E_+^\prime$. For the theory of $uo$, $un$, $uaw$ and $uaw^*$-convergence, we refer to \cite{GX:14,G:14,GTX:16,T:19,KMT:16,Z:16,T:18}.

It can be easily verified that, in $l_p (1\leq p<\infty)$, $uo$, $un$ and $uaw$ and $uaw^*$-convergence of nets are the same as coordinate-wise convergence.  In $L_p(\mu) (1\leq p<\infty)$ for finite measure $\mu$, $uo$-convergence for sequences is the same as almost everywhere convergence, $un$ and $uaw$-convergence for sequences are the same as convergence in measure. In $L_p(\mu) (1< p<\infty)$ for finite measure $\mu$, $uaw^*$-convergence for sequences is also the same as convergence in measure.

In this paper, we characterize the L-weakly compactness in Banach lattices by $uo$, $uaw$ and $uaw^*$-convergence. And as applications, we present some characterizations of L- and M-weakly compact operators on Banach lattices. Some related results are obtained as well.

Recall that a Riesz space $E$ is an ordered vector space in which $x\vee y=\sup\{x,y\}(x\wedge y=\inf\{x,y\})$ exists for every $x,y\in E$. The positive cone of $E$ is denoted by $E_+$, $i.e.,E_+=\{x\in L:x\geq0\}$. For any vector $x$ in $E$ define $x^+:=x\vee0,x^-:=(-x)\vee0, \abs{x}:=x\vee(-x)$. An operator $T:E\rightarrow F$ between two Riesz spaces is said to be \emph{positive} if $Tx\geq0$ for all $x\geq0$. A sequence $(x_n)$ in a Riesz space is called disjoint whenever $n\ne m$ implies $\abs{x_n}\wedge\abs{x_m}=0$ (denoted by $x_n\perp x_m$). A set $A$ in $E$ is said order bounded if there exists some $u\in E_+$ such that $\abs{x}\leq u$ for all $x\in A$. An operator $E:L\rightarrow F$ is called order bounded if it maps order bounded subsets of $E$ to order bounded subsets of $F$. Throughout this paper, $E$ will stand for a Banach lattice. A Banach lattice $E$ is a Banach space $(E, \Vert\cdot\Vert)$ such that $E$ is a Riesz space and its norm satisfies the following property: for each $x, y\in E$ with $\abs{x}\leq\abs{y}$, we have $\Vert x\Vert\leq\Vert y\Vert$.

For undefined terminology, notation and basic theory of Riesz space, Banach lattice and linear operator, we refer to \cite{AB:06,MN:91}.

\section{Results}\label{}
An element $e\in E_+$ is called an \emph{atom} of the Riesz space $E$ if the principal ideal $E_e$ is one-dimensional. $E$ is called an atomic Banach lattice if it is the band generated by its atoms. 
\begin{proposition}\label{}
Let $E$ be a Banach lattice, the following statements hold.
\begin{enumerate}
\item $E$ is atomic and order continuous iff $x_\alpha^\prime\xrightarrow{w^*}0\Rightarrow x_\alpha^\prime\xrightarrow{uaw^*}0$ for any net $(x_\alpha^\prime)$ in $E^\prime$.
\item $E$ is atomic and reflexive iff $x_\alpha^\prime\xrightarrow{uaw^*}0\Leftrightarrow x_\alpha^\prime\xrightarrow{w}0$ for any bounded net $(x_\alpha^\prime)$ in $E^\prime$.
\end{enumerate}
\end{proposition}

\begin{proof}
$(1)\Rightarrow$ For a $w^*$-null net $(x_\alpha^\prime)\subset E^\prime$, it follows from \cite[Theorem~3.4]{G:14} that $x^\prime_\alpha\xrightarrow{uo}0$, therefore $x^\prime_\alpha\xrightarrow{uaw^*}0$.
	
$(1)\Leftarrow$ If $E$ is not order continuous, then there exists a bounded disjoint sequence $(x_n^\prime)$ which is not $w^*$-null by \cite[Corollary~2.4.3]{MN:91}. Therefore, we can find some subnet $(x_\alpha^\prime)$ of $(x_n^\prime)$ which $w^*$-converges to a non-zero functional $x^\prime\neq0$ on $E$. Hence, $(x_\alpha^\prime-x^\prime)\xrightarrow{uaw^*}0$. But, $x_\alpha^\prime\xrightarrow{uaw^*}0$ since $x_n^\prime\xrightarrow{uaw^*}0$. This yields a contradiction.
	
Assume that $E$ is not atomic, then $E^\prime$ is not atomic by \cite[Corollary~2.3]{AE:10}. It follows from \cite[Theorem~3.1]{CW:98} that there exists a sequence $(x_n^\prime)\subset E^\prime$ such that $x_n^\prime\xrightarrow{w^*}0$ and $|x_n^\prime|=x^\prime>0$ for all $n\in N$ and some $x^\prime\in E^\prime$. But, $(x_n^\prime)$ is not $uaw^*$-null since $|x_n^\prime|\wedge x^\prime=x^\prime\nrightarrow0$ by $\sigma(E^\prime,E)$, which is absurd.

$(2)$ by \cite[page~87]{T:18} and $(1)$.
\end{proof}		

Recall that a bounded subset $A$ in a Banach lattice $E$ is said to be \emph{L-weakly compact} if every disjoint sequence in the solid hull of $A$ is convergent to zero. $E^a$ is the maximal ideal in Banach lattice $E$ on which the induced norm is order continuous as
$$E^a=\{x\in E:\text{every monotone sequence in $[0,|x|]$ is convergent}\}.$$
(\cite[Page~71]{MN:91}) Let $E$ be Riesz space, for every non-empty subsets $A\subset E$ and $B\subset E^\sim$, we define the generated absolutely monotone seminorms
$$\rho_{A}(x^\prime)=\sup\{|x^\prime|(|x|):x\in A\}, \rho_{B}(x)=\sup\{|x^\prime|(|x|):x^\prime\in B\},$$ for $x^\prime\in E^\sim$ and $x\in E$.

The following results characterize L-weakly compactness in Banach lattices by $uo$, $uaw$ and $uaw^*$-convergence.

\begin{theorem}\label{L-weakly compact sets}
Let $E$ be a Banach lattice, the following statements hold.
\begin{enumerate}
\item For a non-empty bounded subset $A\subset E$, the following conditions are equivalent.
\begin{enumerate}
\item $A$ is L-weakly compact.
\item\label{L-weakly compact sets-11} $\sup_{x\in A}|x^\prime_n(x)|\rightarrow0$ for every norm bounded $uaw^*$-null sequence $(x^\prime_n)\subset E^\prime$.
\item\label{L-weakly compact sets-22} $\sup_{x\in A}|x^\prime_n(x)|\rightarrow0$ for every norm bounded $uaw$-null sequence $(x^\prime_n)\subset E^\prime$.
\item $\sup_{x\in A}|x^\prime_n(x)|\rightarrow0$ for every norm bounded $uo$-null sequence $(x^\prime_n)\subset E^\prime$.
\end{enumerate}
\item For a non-empty bounded subset $B\subset E^\prime$, $B$ is L-weakly compact iff $\sup_{x^\prime\in B}|x^\prime(x_n)|\rightarrow0$ for every norm bounded $uaw$-null sequence $(x_n)\subset E$.
\end{enumerate}
	
\end{theorem}
\begin{proof}
$(1)(b)\Rightarrow(1)(c)$ Every $uaw$-null net is $uaw^*$-null.

$(1)(b)\Rightarrow(1)(d)$ Every $uo$-null net is $uaw^*$-null.
	
$(1)(d)\Rightarrow(1)(a)$ Let $(x_n^\prime)$ be an arbitrary disjoint sequence in $B_{E^\prime}$. To prove that $A$ is L-weakly compact, we only need to show that $\rho_{A}(x^\prime_n)\rightarrow0$ by \cite[Proposition~3.6.2(2)]{MN:91}, where $\rho_{A}(f)$ is defined by $$\rho_{A}(f)=\sup\{|f|(|x|):x\in A\}=\sup\{|g(x)|:|g|\leq|f|,x\in A\}$$ for every $f\in E^\prime$. Assume by way of contradiction that $\rho_{A}(x^\prime_n)\nrightarrow0$. Then, by passing to a subsequence if necessary, we can suppose that there would exist some $\epsilon>0$ satisfying $\rho_{A}(x^\prime_n)=\sup\{|x^\prime_n|(\abs{x}):x\in A\}>\epsilon$ for all $n$. For each $n$ choose some $x_n\in A$ and some $(y_n^\prime)$ in $E^\prime$ with $|y_n^\prime|\leq |x_n^\prime|$ such that $|y_n^\prime(x_n)|>\epsilon$. Clearly, $(y_n^\prime)$ is likewise a norm bounded disjoint sequence. Hence, $y_n^\prime\xrightarrow{uo}0$. Therefore, $(y_n^\prime)$ converges uniformly to zero on $A$. This leads to a contradiction.
	
$(1)(c)\Rightarrow(1)(a)$ is similar to $(1)(d)\Rightarrow(1)(a)$.
	
$(1)(a)\Rightarrow(1)(b)$ According to \cite[Proposition~3.6.2]{MN:91}, for any $\epsilon>0$, there exists some $x\in E^a_+$ such that $A\subset[-x,x]+\epsilon\cdot B_{E}$. 
	
First of all, we prove that $y_n^\prime(x)\rightarrow0$ for every disjoint sequence $(y_n^\prime)$ in $B_{E^\prime}$. Every monotone sequence in $[0,x]$ is norm convergent since $x\in E^a_+$. According to \cite[Corollary~2.3.6]{MN:91}, every disjoint sequence in $[-x,x]$ is norm convergent. It follows from \cite[Theorem~2.3.3]{MN:91} that   $\rho_{[-x,x]}(x_n^\prime)\rightarrow0$ for every disjoint sequence $(y_n^\prime)\subset B_{E^\prime}$. Therefore, $y_n^\prime(x)\rightarrow0$.
	
Then, we claim that every bounded $uaw^*$-null sequence $(x_n^\prime)$ converges uniformly on $[-x,x]$. We can find a sequence of functionals $(z_n^\prime)$ such that $\abs{y^\prime_n}(\abs{x})=|z_n^\prime(x)|$ with $|z_n^\prime|\leq|y_n^\prime|$ since $\abs{f}(\abs{x})=\max\{|g(x)|:|g|\leq|f|\}$. Clearly, $(z_n^\prime)$ is also disjoint. Hence, $\abs{y^\prime_n}(\abs{x})\rightarrow0$. Applying \cite[Theorem~4.36]{AB:06} to the seminorm be $\abs{\cdot}(\abs{x})$, the identity operator and $B_{E^\prime}$, we have that, for any $\epsilon>0$, there exists some $u^\prime\in E^\prime_+$ such that
$$\sup_{x^\prime\in B_{E^\prime}}(\abs{x^\prime}-\abs{x^\prime}\wedge u^\prime)(\abs{x})=\sup_{x^\prime\in B_{E^\prime}}\big((\abs{x^\prime}-u^\prime)^+\big)(\abs{x})<\epsilon.$$

Clearly, $(\abs{x_n^\prime}\wedge u^\prime)(|x|)\rightarrow0$. Therefore, $\abs{x^\prime_n}(|x|)\rightarrow0$, moreover $\sup_{y\in [-x,x]}\abs{x^\prime_n(y)}\rightarrow0$. 
	
Finally, according to the arbitrariness of $\epsilon$, we have that $\sup_{x\in A}|x^\prime_n(x)|\leq \sup_{x\in [-x,x]+\epsilon\cdot B_{E}}|x^\prime_n(x)|\rightarrow0$.
	
$(2)\Leftarrow$ is similar to $(1)(d)\Rightarrow(1)(a)$.
	
$(2)\Rightarrow$ Using \cite[Proposition~3.6.2]{MN:91} and \cite[Theorem~4.36]{AB:06}, the proof is similar to $(1)(a)\Rightarrow(1)(b)$.
\end{proof}

\begin{corollary}\label{L-weakly compact set-sequence}
Let $E$ be a Banach lattice, $A$ a bounded subset in $E$ and $B$ a bounded subset in $E^\prime$, the following statements hold.
\begin{enumerate}
\item $A$ is L-weakly compact set iff $x_n^\prime(x_n)\rightarrow0$ for every bounded $uaw^*$-null sequence $(x_n^\prime)$ in $E^\prime$ and every sequence $(x_n)$ in $A$ iff $x_n^\prime(x_n)\rightarrow0$ for every bounded $uaw$-null sequence $(x_n^\prime)$ in $E^\prime$ and every sequence $(x_n)$ in $A$ iff $x_n^\prime(x_n)\rightarrow0$ for every bounded $uo$-null sequence $(x_n^\prime)$ in $E^\prime$ and every sequence $(x_n)$ in $A$.
\item $B$ is L-weakly compact set iff $x_n^\prime(x_n)\rightarrow0$ for every bounded $uaw$-null sequence $(x_n)$ in $E$ and every sequence $(x_n^\prime)$ in $B$.
\end{enumerate}
\end{corollary}
  
\begin{theorem}\label{disjoint-d}
Let $E$ be a Banach lattice, for bounded solid subsets $A$ of $E$ and $B$ of $E^\prime$, the following statements hold.
\begin{enumerate}
\item If $E$ has order continuous norm, then the following conditions are equivalent.
\begin{enumerate}
\item $A$ is L-weakly compact set.
\item\label{disjoint-d-1} For every positive disjoint sequence $(x_n)$ in $A$ and each bounded $uaw^*$-null sequence $(x_n^\prime)$ in $E^\prime$, $x_n^\prime(x_n)\rightarrow0$.
\item For every positive disjoint sequence $(x_n)$ in $A$ and each bounded $uaw$-null sequence $(x_n^\prime)$ in $E^\prime$, $x_n^\prime(x_n)\rightarrow0$.
\item For every positive disjoint sequence $(x_n)$ in $A$ and each bounded $uo$-null sequence $(x_n^\prime)$ in $E^\prime$,     $x_n^\prime(x_n)\rightarrow0$.
\end{enumerate}
\item If $E^\prime$ has order continuous norm, then $B$ is L-weakly compact iff $x_n^\prime(x_n)\rightarrow0$ for every positive disjoint sequence $(x_n^\prime)\subset B$ and each bounded $uaw$-null sequence $(x_n)$ in $E$.
\end{enumerate}
	
\end{theorem}

\begin{proof}
$(1)(a)\Rightarrow(1)(b)$ by Corollary \ref{L-weakly compact set-sequence}.
	
$(1)(b)\Rightarrow(1)(a)$ Let $(x_n^\prime)$ be an arbitrary bounded $uaw^*$-null sequence in $E^\prime$. To finish the proof, we have to show that  $\sup_{x\in A}|x_n^\prime(x)|\rightarrow0$. Assume by way of contradiction that $\sup_{x\in A}|x_n^\prime(x)|\nrightarrow0$.  Then, by passing to a subsequence if necessary, we can suppose that there would exist some $\epsilon>0$ such that $\sup_{x\in A}|x_n^\prime(x)|>\epsilon$ for all $n$. Note that the equality $\sup_{x\in A}|x_n^\prime(x)|=\sup_{0\leq x\in A}|x_n^\prime|(x)$ holds, since $A$ is solid. $|x_n^\prime|\xrightarrow{w^*}0$ since $E^\prime$ is order continuous. Let $n_1=1$. Because $|x_n^\prime|(4x_{n_1})\rightarrow0$,
there exists some $1<n_2\in\mathbb{N}$ such that $|x_{n_2}^\prime|(4x_{n_1})<\frac{1}{2}$. It is easy to see that we can find a strictly increasing subsequence $(n_k)_{k=1}^\infty\subset \mathbb{N}$ such that $|x_{n_{m+1}}^\prime|(4^m\sum_{k=1}^{m}x_{n_k})<\frac{1}{m}$ for all $m$. Let
$$x=\sum\limits_{k=1}^\infty2^{-k}x_{n_{k}}, y_m=(x_{n_{m+1}}-4^m\sum\limits_{k=1}^{m}x_{n_k}-2^{-m}x)^+.$$
According to \cite[Lemma~4.35]{AB:06},$(y_m)$ is a disjoint sequence in $A\cap E_+$. Now, we have
\begin{align*}
|x^\prime_{n_{m+1}}|(y_m)&=|x^\prime_{n_{m+1}}|(x_{n_{m+1}}-4^m\sum\limits_{k=1}^{m}x_{n_k}-2^{-m}x)^+\\
&\geq |x^\prime_{n_{m+1}}|(x_{n_{m+1}}-4^m\sum\limits_{k=1}^{m}x_{n_k}-2^{-m}x)\\
&=|x^\prime_{n_{m+1}}|(x_{n_{m+1}} )-|x^\prime_{n_{m+1}}|(4^m\sum\limits_{k=1}^{m}x_{n_k})-2^{-m}|x^\prime_{n_{m+1}}|x\\
&>\epsilon-\frac{1}{m}-2^{-m}|x^\prime_{n_{m+1}}|x.
\end{align*} 
Let $m\rightarrow\infty$, it is clear that $2^{-m}|x^\prime_{n_{m+1}}|x\rightarrow0$. Hence, $|x^\prime_{n_{m+1}}|(y_m)\nrightarrow0$. This leads to a contradiction.
	
The rest of the proof is similar.
\end{proof}

Recall that a continuous operators $T:E\rightarrow F$ from a Banach space to a Banach lattice is said to be \emph{L-weakly compact} whenever $T(B_E)$ is L-weakly compact set in $F$. 

As applications of Theorem \ref{L-weakly compact sets} and Corollary \ref{L-weakly compact set-sequence}, the following results characterize L-weakly compact operators on Banach lattices.
\begin{theorem}\label{L-weakly compact operators}
Let $E$ be a Banach space and $F$ a Banach lattice, the following statements hold.
\begin{enumerate}
\item For a continuous operator $T:E\rightarrow F$, the following conditions are equivalent.
\begin{enumerate}
\item $T$ is L-weakly compact.
\item $y_n^\prime(Tx_n)\rightarrow0$ for each sequence $(x_n)$ in $B_E$ and every bounded $uaw^*$-null sequence $(y_n^\prime)$ in $F^\prime$.	
\item $y_n^\prime(Tx_n)\rightarrow0$ for each sequence $(x_n)$ in $B_E$ and every bounded $uaw$-null sequence $(y_n^\prime)$ in $F^\prime$.	
\item $y_n^\prime(Tx_n)\rightarrow0$ for each sequence $(x_n)$ in $B_E$ and every bounded $uo$-null sequence $(y_n^\prime)$ in $F^\prime$.	
\end{enumerate}
\item For a continuous operator $S:E\rightarrow F^\prime$, $S$ is L-weakly compact iff $S x_n(y_n)\rightarrow0$ for each sequence $(x_n)$ in $B_{E}$ and every bounded $uaw$-null sequence $(x_n)$ in $F$.
\end{enumerate}	
\end{theorem} 
\begin{theorem}\label{L-weakly compact operator-disjoint}
Let $E$ and $F$ be Banach lattices, the following statements hold.
\begin{enumerate}
\item For a positive operator $T:E\rightarrow F$, if $F$ is order continuous, then the following conditions are equivalent.
\begin{enumerate}
\item $T$ is L-weakly compact.
\item $y_n^\prime(Tx_n)\rightarrow0$ for each positive disjoint sequence $(x_n)\subset B_E$ and every bounded $uaw^*$-null sequence $(y_n^\prime)$ in $F^\prime$.	
\item $y_n^\prime(Tx_n)\rightarrow0$ for each positive disjoint sequence $(x_n)\subset B_E$ and every bounded $uaw$-null sequence $(y_n^\prime)$ in $F^\prime$.	
\item $y_n^\prime(Tx_n)\rightarrow0$ for each positive disjoint sequence $(x_n)\subset B_E$ and every bounded $uo$-null sequence $(y_n^\prime)$ in $F^\prime$.	
\end{enumerate}
\item For a positive operator $S:E\rightarrow F^\prime$, if $F^\prime$ is order continuous norm, then $S$ is L-weakly compact iff $S x_n(y_n)\rightarrow0$ for each positive disjoint sequence $(x_n)$ in $B_{E}$ and every bounded $uaw$-null sequence $(x_n)$ in $F$.
\end{enumerate}
	
\end{theorem}
\begin{proof}
$(1)(a)\Rightarrow(1)(b)$ by Theorem \ref{L-weakly compact operators}.
	
$(1)(b)\Rightarrow(1)(a)$ Let $(y_n^\prime)$ be an arbitrary bounded $uaw^*$-null sequence in $F^\prime$. $\abs{y_n^\prime}\xrightarrow{w^*}0$ since $F$ is order continuous. Hence, $|T^\prime(y_n^\prime)|(z)=\sup_{y\in T[-z,z]}|y_n^\prime(y)|\rightarrow0$ for each $z\in E_+$. Without loss of generality, $y_n^\prime\geq0$ for all $n$. To finish the proof, we have to show that $\sup_{x\in B_E}|y_n^\prime(Tx)|\rightarrow0$. Assume by way of contradiction that $\sup_{x\in B_E}|y_n^\prime(Tx)|\nrightarrow0$. Then, by passing to a subsequence if necessary, we can suppose that there would exist some $\epsilon>0$ such that $\sup_{x\in B_E}|y_n^\prime(Tx)|>\epsilon$ for all $n$. Note that the equality  $\sup_{x\in B_E}|y_n^\prime(Tx)|=\sup_{0\leq x\in B_E}\{|y^\prime_n\circ T|(x)\}$ since $A$ is solid. For every $n$, there exists $z_n$ in $B_E\cap E_+$ such that $|T^\prime( y^\prime_n)|(z_n)>\epsilon$. It is similar to the proof of Theorem \ref{disjoint-d} that there exists a subsequence $(y_n)$ of $(z_n)$ and a subsequence $(g_n)$ of $( y_n^\prime)$ such that 
$$|g_n\circ T |(y_n) > \epsilon, |g_{n+1} \circ T |(4^n\sum_{i=1}^{n}y_i)<\frac{1}{n}.$$
Let $x=\sum_{i=1}^{\infty}2^{-i}y_i$ and $x_n=(y_{n+1}-4^n(\sum_{i=1}^{n}y_i)-2^{-n}x)^+$, according to \cite[Lemma~4.35]{AB:06}, $(x_n)$ is positive and disjoint. Hence, 	$$|g_{n+1} \circ T |(x_n)\geq|g_{n+1} \circ T |(y_{n+1}-4^n(\sum_{i=1}^{n}y_i)-2^{-n}x)> \epsilon-\frac{1}{n}-2^{-n}|g_{n+1} \circ T |x.$$
Therefore, $|g_{n+1} \circ T|(x_n)\nrightarrow0$. Clearly, there exists a sequence $(u_n)$ in $E$ satisfying $|u_n|\leq x_n$ such that $|g_{n+1}(Tu_n)|=|g_{n+1} \circ T|(x_n)$. As applications of $$|g_{n+1}(Tu_n^+)|+|g_{n+1}(Tu_n^-)|\geq|g_{n+1}(Tu_n)|=|g_{n+1} \circ T|(x_n)\nrightarrow0,$$ we have $g_{n+1}(Tu_n^+)\nrightarrow0$. This leads to a contradiction.
	
The rest of the proof is similar.
\end{proof}

An operator $T:E\rightarrow F$ from a Banach lattice to a Banach space is said to be \emph{M-weakly compact} if $Tx_n\rightarrow0$ for every disjoint sequence $(x_n)$ in $B_E$.

The following result shows some characterizations of M-weakly compact operators on Banach lattices.
\begin{theorem}\label{M-weakly compact operators}
Let $E$ and $F$ be Banach lattices, for a continuous operator $T:E\rightarrow F$, the following statements hold.
\begin{enumerate}
\item The following conditions are equivalent.
\begin{enumerate}
\item $T$ is a M-weakly compact.
\item $Tx_n\rightarrow0$ for every $uaw$-null sequence $(x_n)$ in $B_E$.
\item $Tx_n\rightarrow0$ for every $uo$-null sequence $(x_n)$ in $B_E$.
\end{enumerate}  
\item For its adjoint operator $T^\prime:F^\prime\rightarrow E^\prime$, the following conditions are equivalent.
		\begin{enumerate}
			\item $T^\prime:F^\prime\rightarrow E^\prime$ is a M-weakly compact operator.
			\item $T^\prime y^\prime_n\rightarrow0$ for every $uaw^*$-null sequence $(y_n^\prime)$ in $B_{F^\prime}$.		
			\item $T^\prime y^\prime_n\rightarrow0$ for every $uaw$-null sequence $(y_n^\prime)$ in $B_{F^\prime}$.
			\item $T^\prime y^\prime_n\rightarrow0$ for every $uo$-null sequence $(y_n^\prime)$ in $B_{F^\prime}$.	
		\end{enumerate}
	\end{enumerate}
	
\end{theorem}
\begin{proof}
$(1)(b)\Leftarrow(1)(a)$ and $(1)(c)\Leftarrow(1)(a)$ Every disjoint sequence is $uaw$ and $uo$-null.
	
$(1)(a)\Rightarrow(1)(b)$ For a bounded $uaw$-null sequence $(x_n)$ in $E$. According to \cite[Proposition~3.6.11]{MN:91}, we can find that $T^\prime:F^\prime\rightarrow E^\prime$ is L-weakly compact. Hence, $\sup_{x^\prime\in T^\prime(B_F^\prime)}|x^\prime(x_n)|\rightarrow0$ for all $uaw$-null sequence $(x_n)$ in $B_E$ by Theorem \ref{L-weakly compact sets}. Therefore, $\Vert Tx_n\Vert=\sup_{x^\prime\in T^\prime(B_F^\prime)}|x^\prime(x_n)|\rightarrow0$ since $\abs{x^\prime(x_n)}=\abs{(T^\prime y^\prime)(x_n)}=\abs{( y^\prime\circ T)(x_n)}$.
	
In other words,                                                                                         according to \cite[Theorem~5.60]{AB:06}, for any $\epsilon>0$, there exists some $u\in E_+$ such that 
$$\Vert T(\abs{x}-\abs{x}\wedge u)\Vert=\Vert T\big((\abs{x}-u)^+\big)\Vert<\epsilon$$ holds for all $x\in B_E$. Clearly, $T$ is order-weakly compact. It follows from \cite[Corollary~3.4.9]{MN:91} that $T(\abs{x_n}\wedge u)\rightarrow 0$. Hence, $T\abs{x_n}\rightarrow0$. Therefore, $Tx_n\rightarrow0$ since $x_n\xrightarrow{uaw}0$ iff $x_n^{\pm}\xrightarrow{uaw}0$ iff $\abs{x_n}\xrightarrow{uaw}0$.

$(1)(a)\Rightarrow(1)(c)$ Using \cite[Theorem~3.4.4]{MN:91}, the proof is similar to $(1)(a)\Rightarrow(1)(b)$.

$(2)$ It is an application of Theorem \ref{L-weakly compact sets} and $\sup_{y\in T(B_E)}|y_n^\prime(y)|=\Vert T^\prime y_n^\prime\Vert$.
\end{proof}
The above results is not true if we replace these convergences to $un$-convergence.
\begin{example}
Every bounded $un$-null sequence in $\big(ba(2^\mathbb N\big))^\prime=((\ell_\infty)^\prime)^\prime$ and $\ell_\infty$ converges uniformly to zero on the unit ball of $ba(2^\mathbb N)=(\ell_\infty){^\prime}$ since $(\ell_\infty)$ and $\big(ba(2^\mathbb N\big))^\prime$ have order unit. But, $B_{ba(2^\mathbb N)}$ is not L-weakly compact.

Every bounded $un$-null sequence in $\big(ba(2^\mathbb N\big))^\prime=((\ell_\infty)^\prime)^\prime$ and $\ell_\infty$ converges uniformly to zero on $I(B_{ba(2^\mathbb N)})$ ($I$ is the identity operator on $B_{ba(2^\mathbb N)}$). But, $I$ is not L-weakly compact.
	
The identity operator $\hat{I}$ on $\ell_\infty$ maps a bounded $un$-null sequence to a norm-null sequence. But, $\hat{I}$ is not M-weakly compact.
\end{example}
\noindent \textbf{Acknowledgement.} The research is supported by National Natural Science Foundation of China(NSFC:51875483).

\end{document}